\documentclass[a4paper,11pt]{amsart}
\usepackage{dsfont,hyperref,fancyhdr,mathrsfs,amsmath,amscd,amsthm,amsfonts,latexsym,amssymb,stmaryrd}
\usepackage[all]{xy}

\DeclareMathOperator{\rank}{rank}

\DeclareMathOperator{\dif}{d}

\DeclareMathOperator{\Lie}{\mathfrak{L}}

\renewcommand{\1}{\mathds{1}} 

\renewcommand{\H}{\mathscr{H}}

\newcommand{\F}{\mathscr{F}}
\newcommand{\Fa}{\mathcal{F}}

\newcommand{\ol}{\mathcal{O}}
\newcommand{\nb}{\mathcal{N}}

\def \a{\alpha}

\def \l{\lambda}

\def \phi{\varphi}
\def \Phi{\varPhi}
\def \p{\pi}
\def \r{\rho}
\def \s{\sigma}

\def \R{\mathbb{R}}
\def \Hq{\mathbb{H}\,}
\def \C{\mathbb{C}\,}

\def\widecheckg{g^{\hspace*{-2.5pt}\vbox to 5pt{\hbox to
0pt{\LARGE$\check{}$}}}\hspace*{2pt}}

\def\widecheckl{\lambda^{\hspace*{-3.5pt}\vbox to 8pt{\hbox to
0pt{\LARGE$\check{}$}}}\hspace*{2pt}}

\begin{document}

\title{On tame $\r$-quaternionic manifolds}
\author{Radu Pantilie}  
\address{R.~Pantilie, Institutul de Matematic\u a ``Simion~Stoilow'' al Academiei Rom\^ane,
C.P. 1-764, 014700, Bucure\c sti, Rom\^ania} 
\email{\href{mailto:Radu.Pantilie@imar.ro}{Radu.Pantilie@imar.ro}} 

\newtheorem{thm}{Theorem}[section]
\newtheorem{lem}[thm]{Lemma}
\newtheorem{cor}[thm]{Corollary}
\newtheorem{prop}[thm]{Proposition}

\theoremstyle{definition}

\newtheorem{defn}[thm]{Definition}
\newtheorem{rem}[thm]{Remark}
\newtheorem{exm}[thm]{Example}

\numberwithin{equation}{section}

\thispagestyle{empty}

\begin{abstract}
We introduce the notion of tame $\r$-quaternionic manifold that permits the construction of a finite family of $\r$-connections, 
significant for the geometry involved. This provides, for example, the following:\\ 
\indent 
$\bullet$ A new simple global characterisation of flat (complex-)quaternionic manifolds.\\ 
\indent 
$\bullet$ A new simple construction of the metric and the corresponding Levi-Civita connection of a quaternionic-K\"ahler manifold 
by starting from its twistor space; moreover, our method provides a natural generalization of this correspondence.\\ 
\indent 
Also, a new construction of quaternionic manifolds is obtained, and the properties of twistorial harmonic morphisms with one-dimensional fibres 
from quaternionic-K\"ahler manifolds are studied. 
\end{abstract}

\maketitle

\section*{Introduction}  

\indent 
A (complex) $\r$-quaternionic manifold $M$ is the parameter space of a locally complete family \cite{Kod} of Riemann spheres embedded into 
a complex manifold $Z$, the twistor space of $M$, with nonnegative normal bundles \cite{Pan-qgfs}\,.\\ 
\indent 
In this paper, we introduce the notion of `tame' $\r$-quaternionic manifold (see Definition \ref{defn:tame_ro_q}\,, below) 
that permits the construction, through the Ward transformation, of a finite family of $\r$-connections, significant for the geometry of $M$ 
(determined by $Z$). For example, if $M$ is quaternionic then only one such `fundamental monopole' exists (with respect to some line bundle over $Z$) 
and its flatness is equivalent to the flatness of $M$ 
(Theorem \ref{thm:quatern_flat}\,). If, further, $Z$ is endowed with a contact structure, our approach leads to a quick simple proof 
(Theorem \ref{thm:contact_qK}\,) of the fact \cite{LeB-qK_twist} that then $M$ is quaternionic-K\"ahler. 
Moreover, our method provides a natural generalization of this fact (Theorem \ref{thm:degen_qK}\,; compare \cite{AleMar-Annali96}\,). 
Finally, this is related to the properties of twistorial harmonic morphisms with one-dimensional fibres 
which we study in Section \ref{section:5}\,, whilst a new construction of quaternionic manifolds is given in Section \ref{section:4}\,.

\section{$\r$-quaternionic manifolds} \label{section:1} 

\indent 
For simplicity, unless otherwise stated, we work in the complex analytic category; also, the manifolds are assumed connected. 
A linear $\r$-quaternionic structure \cite{Pan-qgfs} on a (complex) vector space $U$ is given by a (holomorphic) embedding of the Riemann sphere 
$Y$ into ${\rm Gr}(U)$ such that the corresponding tautological exact sequence 
of vector bundles 
\begin{equation} \label{e:tes} 
0\longrightarrow\Fa\longrightarrow Y\times U\longrightarrow\mathcal{U}\longrightarrow 0 
\end{equation} 
induces an isomorphism between $U$ and the space of sections of $\mathcal{U}$\,. Consequently, 
$H^0(\Fa)$ and $H^1(\Fa)$ are trivial and, thus, $\Fa=L^*\otimes V$, for some vector space $V$, 
where $L$ is any line bundle of Chern number $1$ over $Y$. It follows (see \cite[Proposition 2.5]{Pan-hqo}\,) 
that $$V=H^1\bigl(\bigl(L^2\bigr)^*\bigr)^*\otimes H^0\bigl(L^*\otimes\mathcal{U}\bigr)\;.$$  
Further, let $E$ be the dual of the space of sections of $\Fa^*$. Then 
\begin{equation} \label{e:E} 
E=H^0(L)\otimes H^0\bigl(L^*\otimes\mathcal{U}\bigr)
\end{equation}
is the space of sections of $L\otimes H^0\bigl(L^*\otimes\mathcal{U}\bigr)$\,; in particular, also, $E$ is a $\r$-quaternionic vector space. 
Moreover, dualizing \eqref{e:tes}\,, passing to the cohomology exact sequence and then dualizing again, we obtain 
a linear map $\r:E\to U$ that intertwines the embeddings of the Riemann sphere giving the linear $\r$-quaternionic structures on $E$ and $U$\,; 
thus, the latter is determined by the pair $(E,\r)$ (compare \cite{fq}\,, \cite{fq_2}\,).\\ 
\indent 
The automorphism group of a linear $\r$-quaternionic structure (given by \eqref{e:tes}\,) is the group of extended automorphisms of $\mathcal{U}$ 
(compare \cite{fq}\,). Note that, the automorphism group is the semidirect product of the automorphism group of $\mathcal{U}$ 
and the diffeomorphism group of $Y$. Also, $U$ and $E$ are representation spaces of the automorphism group, 
and \eqref{e:tes} and $\r$ are equivariant.\\  
\indent 
More generally, a $\r$-quaternionic manifold \cite{Pan-qgfs} is given by a diagram 
\begin{equation} \label{e:td}   
\begin{gathered} 
\xymatrix{
       &                         \hspace{1mm}                     Y     \ar[dl]_{\psi} \ar[dr]^{\p}      &                  \\
                      Z                               &                                                                 &            M                             
   } 
\end{gathered} 
\end{equation} 
where $\p:Y\to M$ is a Riemann sphere bundle and $\psi$ is a surjective submersion such that, for any $x\in M$, the restriction of $\psi$ to $\p^{-1}(x)$ 
is a diffeomorphism onto its image $t_x$\,, and the corresponding (surjective) morphism between the normal bundles of $\p^{-1}(x)$ in $Y$ 
and $t_x$ in $Z$ induces an isomorphism between their spaces of sections (compare \cite{Kod}\,). As the normal bundle of $\p^{-1}(x)$ in $Y$ 
can be identified with the trivial vector bundle $\p^{-1}(x)\times T_xM$, the obtained isomorphism between $T_xM$ and the space of sections 
of the normal bundle of $t_x$ in $Z$ gives a linear $\r$-quaternionic structure on $T_xM$, for any $x\in M$.\\ 
\indent 
Note that, $\dif\!\p$ induces an embedding of ${\rm ker}\dif\!\psi$ as a vector subbundle of $\p^*(TM)$\,. We denote by $\mathcal{T}M$ 
the vector bundle over $Y$ which is the quotient of $\p^*(TM)$ through ${\rm ker}\dif\!\psi$\,. Then $TM$ is the direct image by $\p$ of $\mathcal{T}M$.\\ 
\indent 
The usual terminology is that $Z$ is the twistor space of (the $\r$-quaternionic manifold) $M$ and $t_x$ is the twistor sphere corresponding to $x\in M$. 
Note that, from \cite{Kod} it follows that, in a neighbourhood of any of its points, $M$ is determined by the embedding of the corresponding 
twistor sphere in $Z$\,. 

\begin{defn}   
1) We say that $M$ is of \emph{constant type} if all the normal bundles of the twistor spheres have the same isomorphism type.\\ 
\indent 
2) If $Y$ is trivial such that the corresponding projection from $Y$ onto the Riemann sphere factorises into $\psi$ followed by a submersion 
from $Z$ onto the Riemann sphere, we say that $M$ is \emph{$\r$-hypercomplex} (compare \cite{fq}\,, \cite{fq_2}\,, \cite{Pan-qPt}\,, 
and the references therein).
\end{defn} 

\begin{rem} \label{rem:ro-hypercomplex_first} 
Let $M$ be a $\r$-quaternionic manifold. Then the frame bundle of $Y$ is a $\r$-hypercomplex manifold whose twistor space 
is the product of $Z$ and the Riemann sphere.  
\end{rem} 

\indent 
By passing to an open neighbourhood of any point of $M$, we may suppose that there exists a line bundle $L$ over $Z$ such that, for any $x\in M$, 
the restriction of $L$ to $t_x$ has Chern number $1$\,. Let $H$ be the direct image by $\p$ of $\psi^*L$\,. 
Thus, for any $x\in M$, the fibre of $H$ over $x$ is the space of sections of $L|_{\p^{-1}(x)}$\,, and, hence, we have $Y=P(H^*)$\,.\\ 
\indent 
From \eqref{e:E}, it follows that the dual $E$ of the direct image by $\p$ of $({\rm ker}\dif\!\psi)^*$ is $H\otimes F$ 
for some vector bundle $F$ (locally defined) over $M$. Note, however, that $E$ does not depend of $L$\,; consequently, the same holds for 
$\mathcal{E}=L\otimes(\p^*F)$. Furthermore, similarly to the $\r$-quaternionic vector spaces 
case, we have a morphism $\r:E\to TM$ characterising, up to integrability, the $\r$-quaternionic structure on $M$. Note that, 
$\r$ is the direct image of a morphism of vector bundles, over $Y$, from $\mathcal{E}$ to $\mathcal{T}M$. 

\begin{rem} \label{rem:compatible_ro-connection} 
In \cite{Pan-qgfs}\,, it is shown that, locally, there exists a $\r$-connection on $Y$ \emph{compatible} with $\psi$\,. 
By this we mean that there exists a morphism of vector bundles $c:\p^*(E)\to TY$ which when composed with the morphism $TY\to\p^*(TM)$\,,  
induced by $\dif\!\p$, gives $\p^*\!\r$\,, and such that $c\bigl(\{e\otimes f\,|\,f\in F\}\bigr)=({\rm ker}\dif\!\psi)_{{\rm Ann}\,e}$\,, 
for any $e\in H\setminus0$\,, where ${\rm Ann}\,e$ denotes the annihilator of $e$ in $H^*$.\\ 
\indent 
Furthermore, the compatible (local) $\r$-connections on $Y$ form an affine space over the space of sections of $E^*$.\\ 
\indent  
Also, the choice of $L$ such that $H$ is the direct image by $\p$ of $\psi^*L$\,, and a compatible $\r$-connection $c$ determines 
a $\r$-connection $\nabla^H$ on $H$ whose projectivisation is $c$ (compare \cite{War-graviton}\,). This is related, 
through the Ward transformation (see \cite{Pan-hmhKPW}\,), to the fact that 
any two $\r$-connections on $H$ inducing the same compatible $\r$-connection on $Y$ `differ' 
by an anti-self-dual $\r$-connection on the trivial line bundle over $M$ 
(recall \cite{Pan-hmhKPW} that, anti-self-duality for a $\r$-connection 
means that its `restrictions' to the projections by $\p$ of the fibres of $\psi$ are flat connections). 
\end{rem} 

\begin{prop} \label{prop:ro-hyper_to_follow} 
$H$ is a $\r$-quaternionic manifold whose twistor space is $L$\,. 
\end{prop} 
\begin{proof} 
This follows from the following (commutative) diagram 
\begin{equation}  \label{diagram:for_ro-quaternionic} 
\begin{gathered} 
\xymatrix{
                                                       &              H+Y \ar[dl]_{\widetilde{\psi}} \ar[d]  \ar[dr]       &                  \\
    L             \ar[d]   &           \hspace{1mm}            Y    \ar[dl]_{\psi} \ar[dr]^{\p}      &    H \ar[d]              \\
                      Z                               &                                                                 &            M                             
   } 
\end{gathered} 
\end{equation} 
where $H+Y$ is the pull back by $\p$ of $H$, and $\widetilde{\psi}$ 
is induced by the projection $H+Y\to\psi^*L$ given by the fact that $H$ is the direct image by $\p$ of~$\psi^*L$\,. 
\end{proof} 

\begin{cor}[compare \cite{Pan-qPt}\,] \label{cor:ro-hyper_from_ro-quatern} 
Let $P$ be the frame bundle of $H$. Then $P$ is endowed with a ${\rm GL}(2)$-invariant $\r$-hypercomplex structure whose twistor space 
is $(L\oplus L)\setminus0$\,. 
\end{cor}
\begin{proof} 
A quick consequence of Proposition \ref{prop:ro-hyper_to_follow} is that ${\rm Hom}\bigl(\C^{\!2},H\bigr)$ is endowed with a ${\rm GL}(2)$-invariant 
$\r$-quaternionic structure whose twistor space is ${\rm Hom}\bigl(\C^{\!2},L\bigr)$. Now, we may embed $P$ into ${\rm Hom}\bigl(\C^{\!2},H\bigr)$ 
as the open subset which is the union of the free orbits of the action of ${\rm GL}(2)$\,. Consequently, $P$ is endowed with a ${\rm GL}(2)$-invariant 
$\r$-quaternionic structure whose twistor space is $(L\oplus L)\setminus0$\,.\\ 
\indent 
To complete the proof, note that, as $Y$ is a bundle associated to $P$ its pull back by the projection $\p_P:P\to M$ is trivial. 
Furthermore, the projection from $\p_P^*Y=P\times\C\!P^1$ to $\C\!P^1$ factorises as the restriction to $\p_P^*Y(=\p^*P)$ of the morphism of vector bundles 
$\p^*\bigl({\rm Hom}\bigl(\C^{\!2},H\bigr)\bigr)={\rm Hom}\bigl(\C^{\!2},\p^*H\bigr)\to{\rm Hom}\bigl(\C^{\!2},(\psi^*L)\bigr)$\,, over $Y$, followed by the projection 
${\rm Hom}\bigl(\C^{\!2},(\psi^*L)\bigr)\setminus0\to P\bigl({\rm Hom}\bigl(\C^{\!2},(\psi^*L)\bigr)\bigr)=Y\times\C\!P^1$, 
followed by $Y\times\C\!P^1\to\C\!P^1$. The proof quickly follows.  
\end{proof} 

\begin{rem} \label{rem:ro-hyper_from_ro-quatern}
With the same notations as in Corollary \ref{cor:ro-hyper_from_ro-quatern}\,, the pull back of $L$ through the projection from the twistor space of $P$ 
to $Z$ is canonically isomorphic to the pull back through the projection $(L\oplus L)\setminus0\to\C\!P^1$ of the dual of the tautological line bundle. 
Indeed, this can be seen by restricting to any twistor sphere in $Z$. 
\end{rem} 

\begin{prop} \label{prop:for_global_ro_M} 
Let $L$ and $L'$ be line bundles over $Z$ whose restrictions to each twistor sphere have Chern numbers $1$\,.\\ 
\indent 
If there exists $k\in\mathbb{N}\setminus\{0\}$ such that $L^k$ and $(L')^k$ are isomorphic then there exists a brackets preserving isomorphism 
between $TP/{\rm GL}(2)$ and $TP'/{\rm GL}(2)$\,, where $P$ and $P'$ are the frame bundles of the direct images by $\p$ of $\psi^*L$ and $\psi^*L'$, 
respectively.  
\end{prop}  
\begin{proof} 
Let $H$ and $H'$ be the direct images by $\p$ of $\psi^*L$ and $\psi^*L'$, respectively. Then $H'=\mathcal{L}\otimes H$, where $\mathcal{L}$ 
is the line bundle over $M$ corresponding to $L'\otimes L^*$ through the Ward transformation.\\ 
\indent 
If there exists $k\in\mathbb{N}\setminus\{0\}$ such that $L^k$ and $(L')^k$ are isomorphic then $\mathcal{L}^k$ admits a nowhere zero section. 
Consequently, with respect to suitable open coverings, $\mathcal{L}$ is given by locally constant cocycles. The proof quickly follows.   
\end{proof} 

\begin{thm} \label{thm:global_ro_M} 
Let $M$ be a $\r$-quaternionic manifold, and let $\mathcal{L}$ be any line bundle (globally defined) on the twistor space $Z$ of $M$ 
whose restriction to some (and, hence, any) twistor sphere is nontrivial; denote $\mathscr{T}Z=T(\mathcal{L}\setminus0)/(\C\!\setminus\{0\})$\,.\\ 
\indent 
Then, for any twistor sphere $t\subseteq Z$, there exists an exact sequence 
$$0\longrightarrow (L\oplus L)|_t\longrightarrow(\mathscr{T}Z)|_t\longrightarrow Nt\longrightarrow0\;$$  
where $Nt$ is the normal bundle 
of $t$ in $Z$ and $L$ is any line bundle defined in some open neighbourhood of $t$ such that $L|_t$ has Chern number $1$ and 
$\mathcal{L}=L^k$\,, for some $k\in\mathbb{Z}\setminus\{0\}$\,.\\ 
\indent 
Moreover, the direct image by $\p$ of $\psi^*(\mathscr{T}Z)$ is $TP/{\rm GL}(2)$\,, 
where $P$ is the frame bundle of the direct image by $\p$ of $\psi^*L$\,.  
\end{thm} 
\noindent 
(Note that, we can take for $\mathcal{L}$ the anticanonical line bundle of $Z$ and, as the obstruction to the existence of $L$ is topological, 
by passing to an open neighbourhood of each point of $M$, we can always find a line bundle $L$\,, as in the statement of Theorem \ref{thm:global_ro_M}\,.) 
\begin{proof}  
Note that, $\mathscr{T}Z=T(L\setminus0)/(\C\!\setminus\{0\})$\,. Also, the action of ${\rm GL}(2)$ on $P$ is twistorial, 
corresponding to its obvious right action on the twistor space $Z_P=(L\oplus L)\setminus0$ of $P$. 
Consequently, for the second statement, it is sufficient to prove that $\psi^*(\mathscr{T}Z)=\mathcal{T}P/{\rm GL}(2)$\,.\\ 
\indent 
As $Z_P=(L\oplus L)\setminus0$ we have a projection from it onto $Z\times\C\!P^1$ (the projectivisation of $L\oplus L$).  
Let $\p_1$ and $\p_2$ be the projections onto the corresponding factors of $Z\times\C\!P^1$. 
Then $Z_P=(\p_1^{\,*}L)\otimes\p_2^{\,*}\bigl(\ol(-1)\bigr)$\,, where $\ol(-1)$ is the tautological line bundle over $\C\!P^1$. 
Therefore, for any $\ell\in\C\!P^1$, the fibre of the projection from $Z_P$ onto $\C\!P^1$ is $(L\otimes\ell)\setminus0$\,.  
But this is the same with the image of the corresponding map from $P$ to $Z_P$ (the existence of these maps is a consequence 
of Corollary \ref{cor:ro-hyper_from_ro-quatern}\,).  Denote by $\psi_{\ell}:P\to(L\otimes\ell)\setminus0$ the obtained map, 
and, note that, for any $a\in{\rm GL}(2)$\,, we have 
$\psi_{\ell}\circ R_a=R_a\circ\psi_{a(\ell)}$\,, where $R_a$ denotes the (right) translation by $a$ on $P$ and $Z_P$\,.\\ 
\indent 
Now, note that, $\mathcal{T}P$ is the bundle over $P\times\C\!P^1$ whose restriction to $P\times\{\ell\}$\,, for any $\ell\in\C\!P^1$, 
is the quotient of $TP$ through ${\rm ker}\dif\!\psi_{\ell}$ (where we have identified $P$ and $P\times\{\ell\}$\,).  
Equivalently, $\mathcal{T}P$ is characterised by the fact that its restriction to $P\times\{\ell\}$\,, for any $\ell\in\C\!P^1$, 
is the pull back by $\psi_{\ell}$ of the tangent bundle of $(L\otimes\ell)\setminus\{0\}$\,. It follows that 
$\mathcal{T}P/{\rm GL}(2)=\psi^*\bigl(T(L\setminus0)/(\C\!\setminus\{0\})\bigr)$ and the proof is complete. 
\end{proof}  

\begin{rem} \label{rem:about_P}
1) With the same notations as in Theorem \ref{thm:global_ro_M}\,, we, also, have that $TP/{\rm GL}(2)$ is a $\r$-quaternionic manifold 
and its twistor space is $\mathscr{T}Z$ (compare \cite{Pan-q_integrab} and the references therein).\\ 
\indent 
2) Let $\r_P:E_P\to TP$ be the morphism of vector bundles giving (up to integrability) the $\r$-quaternionic structure of $P$.  
The ${\rm GL}(2)$-invariance of the $\r$-hypercomplex structure of $P$ implies that we can pass to quotients, thus, 
obtaining a morphism $\widetilde{E}\to\widetilde{T}M$ of vector bundles over $M$, 
where $\widetilde{E}=E_P/{\rm GL}(2)$ and $\widetilde{T}M=TP/{\rm GL}(2)$\,. Furthermore, 
$\widetilde{E}=H\otimes\widetilde{F}$ for some vector bundle $\widetilde{F}$, and we denote by $\r_M$ 
the canonical morphism of vector bundles from $\widetilde{T}M$ onto $TM$ (with kernel ${\rm Ad}P$). 
Consequently, we have the following diagram 
\begin{equation}  \label{diagram:tilde} 
\begin{gathered} 
\xymatrix{ 
0 \ar[r]     &     {\rm End}H  \ar[d]_{=} \ar[r]      &     \widetilde{E}   \ar[d] \ar[r]     &    E \ar[d]^{\r} \ar[r]      &    0              \\
0 \ar[r]     &     {\rm Ad}P    \ar[r]                  &        \widetilde{T}M   \ar[r]_{\r_M}     &     TM   \ar[r]     &          0                           
   } 
\end{gathered} 
\end{equation} 
where recall that $E=H\otimes F$, with $\r:E\to TM$ giving the $\r$-quaternionic structure of $M$. 
Furthermore, we have a canonical exact sequence 
\begin{equation} \label{e:for_tilde} 
0\longrightarrow H^*\longrightarrow\widetilde{F}\longrightarrow F\longrightarrow 0\;, 
\end{equation}   
which when tensorised with $H$ gives the first row of \eqref{diagram:tilde}\,.\\ 
\indent 
3) Facts contained in this section, up to now (in particular, Theorem \ref{thm:global_ro_M}\,), admit more general formulations. 
The degree of generality, chosen by us, is due to the purposes of this paper. 
\end{rem}

\section{On tame $\r$-quaternionic manifolds} \label{section:2} 

\indent 
In this section, the notations are as in Section \ref{section:1}\,. We start with a useful result which is, also, interesting in itself, 
as it can be seen as an extension of both the Ward transformation and the filtration given by the Birkhoff--Grothendieck theorem. 

\begin{prop} \label{prop:extended_Ward} 
Let $M$ be a $\r$-quaternionic manifold with twistor space $Z$. Let $\Fa$ be a vector bundle over $Z$ whose restriction 
to each twistor sphere has the same isomorphism type $\bigoplus_{j=1}^{j=k}a_j\ol(n_j)$\,, for some $k\in\mathbb{N}\setminus\{0\}$\,, 
and $a_j\in\mathbb{N}\setminus\{0\}$\,, $n_j\in\mathbb{Z}$\,, $j=1,\ldots,k$\,, with $n_j\geq n_{j+1}+2$\,, for any $j=1,\ldots,k-1$\,.\\ 
\indent 
Then there exists a unique filtration $\Fa^1\subseteq\cdots\subseteq\Fa^k=\Fa$ whose restriction to each twistor sphere $t\subseteq Z$ 
is the filtration of $\Fa|_t$ given by the Birkhoff--Grothendieck theorem.  
\end{prop}  
\begin{proof} 
As it is sufficient to prove that this holds in an open neighbourhood of each twistor sphere, we may suppose that there exists a line bundle $L$ 
over $Z$ whose restriction to each twistor sphere has Chern number $1$\,. Further, by tensorising $\Fa$ with a power of $L$\,, if necessary, 
we may suppose $n_1=0$\,, and, thus, $n_j\leq-2$\,, for any $j=2,\ldots,k$\,.\\ 
\indent 
Now, $\psi^{*\!}\Fa$ is endowed with a flat partial connection over ${\rm ker}\dif\!\psi$\,, 
and admits a filtration $\widetilde{\Fa}^1\subseteq\cdots\subseteq\widetilde{\Fa}^k=\psi^{*\!}\Fa$ whose restriction to each fibre of $\p$  
is the filtration of the restriction of $\psi^{*\!}\Fa$ to that fibre, given by the Birkhoff--Grothendieck theorem.\\ 
\indent 
As $n_j\leq-2$\,, for any $j=2,\ldots,k$\,, the direct image by $\p$ of the partial connection, over ${\rm ker}\dif\!\psi$\,, on $\psi^{*\!}\Fa$ 
gives an anti-self-dual $\r$-connection on the direct image by $\p$ of $\widetilde{\Fa}^1$ which corresponds, through the Ward transformation, 
with a vector subbundle $\Fa^1$ of $\Fa$. The proof quickly follows, inductively.  
\end{proof} 

\begin{defn} \label{defn:tame_ro_q} 
Let $M$ be a $\r$-quaternionic manifold and let $\mathcal{L}$ be any line bundle on the twistor space $Z$ of $M$ 
whose restriction to some twistor sphere is nontrivial; denote $\mathscr{T}Z=T(\mathcal{L}\setminus0)/(\C\!\setminus\{0\})$\,.\\ 
\indent 
We say that $M$ is \emph{tame, with respect to $\mathcal{L}$}\,, if there exists a filtration 
$\mathscr{T}^1Z\subseteq\cdots\subseteq\mathscr{T}^kZ=\mathscr{T}Z$ 
whose restrictiction to each twistor sphere $t\subseteq Z$ is the filtration of $\mathscr{T}Z|_t$ given by the Birkhoff--Grothendieck theorem. 
\end{defn} 

\indent 
Let $M$ be a tame $\r$-quaternionic manifold, with respect to some line bundle $\mathcal{L}$ over $Z$. 
By Theorem \ref{thm:global_ro_M} the filtration of $\mathscr{T}Z$ corresponds to a filtration (compare \cite{fq}\,) 
$0=\widetilde{T}^0M\subseteq\cdots\subseteq\widetilde{T}^lM=\widetilde{T}M$, for some $l\in\mathbb{N}$\,, where recall that 
$\widetilde{T}M=TP/{\rm GL}(2)$\,, with $P$ is the frame bundle of the direct image by $\p$ of $\psi^*L$, where $L$ is any line bundle, 
locally defined over $Z$, whose restrictions to the twistor spheres have Chern number $1$ and such that 
$\mathcal{L}=L^p$\,, for some $p\in\mathbb{Z}\setminus\{0\}$\,. In particular, $P$ is of constant type.\\ 
\indent 
Furthermore, there exists a unique decreasing sequence of natural numbers 
$p_j$\,, $j=1,\ldots,l$\,, such that $\widetilde{T}^jM/\widetilde{T}^{j-1}M=\bigl(\odot^{p_j}H\bigr)\otimes\widetilde{F}_j$\,, 
for some vector bundles $\widetilde{F}_j$\,, $j=1,\ldots,l$\,, where $\odot$ denotes the symmetric power.\\ 

\begin{defn} \label{defn:fundamental_monopoles} 
Let $M$ be a $\r$-quaternionic manifold which is tame with respect to some line bundle $\mathcal{L}$ over its twistor space $Z$.\\ 
\indent  
The anti-self-dual $\r$-connections on $\widetilde{F}_j$ obtained by applying the Ward transformation to 
$L^{-p_j}\otimes\bigl(\mathscr{T}^jZ/\mathscr{T}^{j-1}Z\bigr)$, 
$j=1,\ldots,l$, are called the \emph{fundamental mo\-no\-poles of $M$, with respect to $\mathcal{L}$}\,. 
\end{defn} 
 
\indent 
Suppose now that $M$ is a $\r$-quaternionic manifold of constant type. 
The canonical filtration given by the Birkhoff--Grothendieck theorem gives 
the (increasing) \emph{canonical filtration} of the tangent bundle of $M$\,:
$0=T^0M\subseteq\cdots\subseteq T^kM=TM$, where $k\in\mathbb{N}$\,. Accordingly, 
we have a canonical filtration $0=E^0\subseteq\cdots\subseteq E^k=E$ such that $\r\bigl(E^j\bigr)\subseteq T^jM$, for $j=0,\ldots,k$\,. 
Note that, the canonical filtration of $E$ is increasing if and only if $\r:E\to TM$ is surjective; 
equivalently, the Birkhoff--Grothendieck decomposition of the normal bundles of the twistor spheres contains no trivial term 
(more precisely, only positive Chern numbers are appearing in that decomposition).\\ 
\indent 
With $L$ a line bundle (locally defined) on $Z$ whose restriction to each twistor sphere has Chern number $1$\,, and 
assuming $k\in\mathbb{N}\setminus\{0\}$ (equivalently, $M$ of positive dimension), there exists a unique decreasing sequence of natural numbers 
$n_j$\,, $j=1,\ldots,k$\,, such that $T^jM/T^{j-1}M=\bigl(\odot^{n_j}H\bigr)\otimes F_j$\,, 
for some vector bundles $F_j$\,, $j=1,\ldots,k$\,.\\ 
\indent 
If the normal exact sequeces of the twistor spheres split then $\widetilde{T}M$ admits an increasing filtration with $k$ or $k+1$ terms, 
according to whether or not there exists $j_0\in\{k-1,k\}$ such that $n_{j_0}=1$\,. 
Moreover, for any $j\in\{1,\ldots,k\}\setminus\{j_0\}$ the corresponding term of the filtration of $\widetilde{T}M$ is given by $T^jM$\,; 
in particular, $p_j=n_j$ and $\widetilde{F}_j=F_j$\,. Note that, if $0\in\{n_j\}_{j=1,\ldots,k}$ then $\widetilde{T}^{k+1}M/\widetilde{T}^kM=T^kM/T^{k-1}M$.  

\begin{rem}  
1) The canonical filtrations of $TM$ and $E$\,, and, if $k$ is nonzero, the natural numbers $n_j$\,, $j=1,\ldots,k$\,, are independent of $L$\,.\\ 
\indent 
2) By Proposition \ref{prop:for_global_ro_M} and Theorem \ref{thm:global_ro_M} we may, for example, assume $TP_j/{\rm GL}(r_j)$ globally defined 
on $M$, where $P_j$ is the frame bundle of $F_j$ and $r_j=\rank F_j$\,. Therefore, by a slight abuse of terminology, we may speak of $\r$-connections 
on $F_j$ (although $F_j$ may be defined only locally on $M$).\\ 
\indent 
3) If there exists $j_0\in\{k-1,k\}$ such that $n_{j_0}=1$ then 
$\widetilde{T}^{j_0}M/\widetilde{T}^{j_0-1}M=H\otimes\widetilde{F}_{j_0}$ for some vector bundle $\widetilde{F}_{j_0}$ 
and we have an exact sequence  $$0\longrightarrow H^*\longrightarrow \widetilde{F}_{j_0}\longrightarrow F_{j_0}\longrightarrow 0$$ 
which splits if and only if $H$ admits a connection over $M$.    
\end{rem}  

\begin{cor} \label{cor:fundamental_monopoles} 
Let $M$ be a $\r$-quaternionic manifold of constant type, and let $\mathcal{L}$ be any line bundle on the twistor space $Z$ of $M$ 
whose restriction to some twistor sphere is nontrivial; denote $\mathscr{T}Z=T(\mathcal{L}\setminus0)/(\C\!\setminus\{0\})$\,.\\ 
\indent 
Suppose that $M$ is tame, with respect to $\mathcal{L}$\,, and the normal exact sequences of the twistor spheres split.\\  
\indent 
\quad{\rm (i)} If there exists $j_0\in\{k-1,k\}$ such that $n_{j_0}=1$ then $\widetilde{F}_{j_0}$ and $F_j$\,, for any $j\in\{1,\ldots,k\}\setminus\{j_0\}$\,,  
are endowed with anti-self-dual $\r$-connections.\\ 
\indent 
\quad{\rm (ii)} If there does not exist $j_0\in\{k-1,k\}$ such that $n_{j_0}=1$ then $H$ and $F_j$\,, for any $j\in\{1,\ldots,k\}$\,, 
are endowed with anti-self-dual $\r$-connections.  
\end{cor} 
\begin{proof} 
This is a consequence of the Ward transformation and Theorem \ref{thm:global_ro_M}\,. 
\end{proof} 

\subsection{On $\r$-hypercomplex manifolds of constant type} In this subsection $M$ is $\r$-hypercomplex.  
Then $Y=M\times\C\!P^1$ such that the projection from $Y$ onto $\C\!P^1$ factorises into $\psi$ followed by a surjective submersion $\chi:Z\to\C\!P^1$. 
For any $z\in\C\!P^1$, we denote by $\psi_z$ the restriction of $\psi$ to $M\times\{z\}$\,. Thus, for any $z\in\C\!P^1$, on identifying $M$ with $M\times\{z\}$\,, 
we have that $\psi_z:M\to N_z$ is a surjective submersion, where $N_z=\chi^{-1}(z)$ (compare the proof of Theorem \ref{thm:global_ro_M}\,). 
Let $L$ be the pull back by $\chi$ of the dual of the tautological bundle over $\C\!P^1$, and denote $\mathscr{T}Z=T(L\setminus0)/(\C\!\setminus\{0\})$\,. 

\begin{cor} \label{cor:BirGro} 
Suppose that, for some $k\in\mathbb{N}\setminus\{0\}$\,, there exist positive natural numbers $n_1>\ldots>n_k$ and 
an increasing filtration $0=\nb^0\subseteq\cdots\subseteq\nb^k={\rm ker}\dif\!\chi$ over $Z$ such that 
$L^{-n_j}\otimes\bigl(\nb^j/\nb^{j-1}\bigr)$ restricted to any twistor sphere is trivial, $j=1,\ldots,k$\,.\\ 
\indent 
Then $M$ is tame, with respect to $L$, and its fundamental monopoles are (essentially) obtained by applying the Ward transformation 
to $L^{-n_j}\otimes\bigl(\nb^j/\nb^{j-1}\bigr)$\,, $j=1,\ldots,k$\,. 
\end{cor} 
\begin{proof} 
Note that, the normal bundle of each twistor sphere $t\subseteq Z$ is isomorphic to $({\rm ker}\dif\!\chi)|_t$\,. 
Also, we have an exact sequence $0\longrightarrow{\rm ker}\dif\!\chi\longrightarrow\mathscr{T}Z\longrightarrow\chi^*(L\oplus L)\longrightarrow0$\,, 
induced by $\dif\!\chi$\,. Consequently, the filtration of ${\rm ker}\dif\!\chi$ determines a filtration of $\mathscr{T}Z$ as required.  
\end{proof} 
 
\begin{rem} 
1) If there exists $j_0\in\{k-1,k\}$ such that $n_{j_0}=1$ then 
$\widetilde{F}_{j_0}=H^*\oplus F_{j_0}$ (because $H$ admits the trivial flat connection) and its anti-self-dual $\r$-connection 
is, accordingly, a direct sum.\\ 
\indent 
2) The filtration of ${\rm ker}\dif\!\psi$ corresponds to increasing filtrations 
$$0=T^0N_z\subseteq T^1N_z\subseteq\cdots\subseteq T^kN_z=TN_z$$ and, consequently, to isomorphisms $F_j=\psi_z^*\bigl(T^jN_z/T^{j-1}N_z\bigr)$ 
under which the anti-self-dual $\r$-connections of $F_j$\,, $j=1,\ldots,k$\,, are induced by the partial Bott connections of the foliations determined 
by $\psi_z$\,, $z\in\C\!P^1$ (compare \cite{Pan-hmhKPW}\,).\\ 
\indent  
3) If $k=1$ then $TM=(\odot^{n_1}H)\otimes F_1$ and the tensor product of the trivial connection and the anti-self-dual $\r$-connection of $F_1$ 
is just the Obata $\r$-connection from \cite{Pan-hmhKPW}\,.\\ 
\indent 
4) With the same notations as in Corollary \ref{cor:ro-hyper_from_ro-quatern}\,, from Remark \ref{rem:ro-hyper_from_ro-quatern} 
and Corollary \ref{cor:BirGro} it follows that for any tame $\r$-quaternionic manifold whose twistor spheres have positive normal bundles, 
also, the $\r$-hypercomplex manifold $P$ is tame.\\ 
\indent 
5) If $Z$ is surface then obviously $M$ is of constant type. The case $n_1=2$ was considered in \cite{Hit-complexmfds}\,. With our approach, 
the relevant Weyl connection can be quickly obtained by building, firstly, a suitable bracket (see \cite{Pan-qgfs} and the references therein) on~$E$.\\ 
\indent 
If $n_1\in\mathbb{N}$ and $M$ is $\r$-hypercomplex then $M$ is tame with one possibly nontrivial fundamental monopole,  
determining the Obata $\r$-connection \cite{Pan-hmhKPW} of $M$\,; the case $n_1=2$ was considered in \cite{GauTod}\,.     
\end{rem}

\section{Quaternionic manifolds} \label{section:3} 

\indent 
The quaternionic manifolds are characterised, among the $\r$-quaternionic manifolds, 
by the fact that the Birkhoff--Grothendieck decompositions 
of the normal bundles of the twistor spheres contain only terms of Chern number $1$\,. Consequently, they are tame $\r$-qua\-ter\-ni\-onic manifolds 
with only one fundamental monopole. Also, note that, the dimension of any quaternionic manifold is even. Furthermore, for any line bundle, 
over the twistor space of a quaternionic manifold, whose restriction to some twistor sphere is nontrivial, the corresponding fundamental monopole 
is a (classical) connection.\\ 
\indent 
The `flat model' is the Grassmannian ${\rm Gr}_2(n+2)$ with twistor space $Z=\C\!P^{n+1}$ and $Y$ the flag manifold $F_{1,2}(n+2)$\,. In this case,  
the fundamental monopole, with respect to any nontrivial line bundle $\mathcal{L}$ on $Z$, (essentially) is the trivial connection on 
${\rm Gr}_2(n+2)\times\C^{\!n+2}$. 

\begin{thm}[compare \cite{Sal-dg_qm}\,,\,\cite{Pan-proj_ro}\,] \label{thm:quatern_flat} 
Let $M$ be a quaternionic manifold, $\dim M=2n$\,, and let $Z$ be its twistor space. Then the following assertions are equivalent:\\ 
\indent 
{\rm (i)} There exists a line bundle $\mathcal{L}$ over $Z$ whose restriction to some twistor sphere is nontrivial and 
such that the fundamental monopole of $M$, with respect to $\mathcal{L}$\,, is flat.\\ 
\indent 
{\rm (ii)} There exists a twistorial local diffeomorphism from a covering space of $M$ to ${\rm Gr}_2(n+2)$\,. 
\end{thm} 
\begin{proof} 
Let $\mathcal{L}$ be a line bundle over $Z$ whose restriction to some twistor sphere is nontrivial. 
With the same notations as in \eqref{e:for_tilde} and Corollary \ref{cor:fundamental_monopoles}\,, we have $j_0=k=1$ 
and $\widetilde{F}=\widetilde{F}_1$\,. Also, although $\widetilde{F}$ may exist only locally, ${\rm Gr}_2\bigl(\widetilde{F}\bigr)$ is globally defined 
over $M$, and the locally defined embeddings $H^*\subseteq\widetilde{F}$, given by \eqref{e:for_tilde}\,, define a (global) section $\s$ of 
${\rm Gr}_2\bigl(\widetilde{F}\bigr)$\,.\\ 
\indent 
The fundamental monopole of $M$, with respect to $\mathcal{L}$\,, induces a connection on ${\rm Gr}_2\bigl(\widetilde{F}\bigr)$ which, if (i) holds, 
is flat. Thus, assuming (i)\,, by passing to a covering space of $M$, if necessary, we have ${\rm Gr}_2\bigl(\widetilde{F}\bigr)=M\times{\rm Gr}_2(V)$\,, 
where $V$ is a vector space of dimension $n+2$\,. Hence, $\s_x=(x,\phi(x))$\,, for any $x\in M$\,, for some map $\phi:M\to{\rm Gr}_2(V)$\,.\\ 
\indent 
To complete the proof we have to show that $\phi$ is a twistorial local diffeomorphism. For this, by passing to an open neighbourhood of each point of $M$, 
if necessary, we may suppose that $\mathcal{L}$ restricted to any twistor sphere has Chern number $1$\,. Then the flatness of the fundamental monopole, 
with respect to $\mathcal{L}$\,, is equivalent to the existence of an isomorphism of vector bundles $\a:\mathscr{T}Z\to\mathcal{L}\otimes V$, where 
$\mathscr{T}Z=T(\mathcal{L}^*\setminus0)/(\C\!\setminus0)$\,.\\ 
\indent 
As $\mathcal{L}$ has rank $1$\,, the adjoint bundle of $\mathcal{L}^*\setminus0$ is trivial. Thus, we have an embedding $Z\times\C\subseteq\mathscr{T}Z$. 
Denote by $\1$ the section of $\mathscr{T}Z$ induced through this embedding by the section $x\mapsto(x,1)$ of $Z\times\C$.\\ 
\indent 
Then $s=\a(\1)$ is a nowhere zero section of $\mathcal{L}\otimes V$ which induces a section of $P(\mathcal{L}\otimes V)=Z\times PV$, 
obviously, given by $z\mapsto(z,\phi_Z(z))$\,, for some map $\phi_Z:Z\to PV$.\\       
\indent 
We claim that $\phi_Z$ is a local diffeomorphism. Indeed, the differential of $\phi_Z$ is determined by $\nabla s$ (see \cite{Pan-proj_ro}\,), 
where $\nabla$ is the tensor product of the canonical $\r_Z$-connection on $\mathcal{L}$ and the trivial connection on $Z\times V$, 
where $\r_Z:\mathscr{T}Z\to TZ$ is the projection. Thus, also, denoting by $\nabla$ the $\r_Z$-connection on $\mathscr{T}Z$ with respect to 
which $\a$ is covariantly constant, we have to show that $\nabla\1$ is an isomorphism, at each point.\\ 
\indent 
The existence of $\a$ shows that the frame bundle of $\mathscr{T}Z$ admits a reduction to $\C\setminus\{0\}$\,, embedded into ${\rm GL}(V)$ 
through $\l\mapsto\l\,{\rm Id}_V$\,. Furthermore, $\nabla$ is the flat $\r_Z$-connection corresponding to this reduction. Hence, 
$\nabla\1={\rm Id}_{\mathscr{T}Z}$ which implies that $\phi_Z$ is a local diffeomorphism. Moreover, as $\mathcal{L}$ restricted to each 
twistor sphere has Chern number $1$\,, we have that $\phi_Z$ maps each twistor sphere diffeomorphically onto a projective line of $PV$, 
and the proof quickly follows.    
\end{proof} 

\indent 
Let $M$ be a quaternionic manifold with twistor space $Z$ given by $\psi:Y\to Z$ and $\p:Y\to Z$. By passing to an open neighbourhood of each point, 
if necessary, let $L$ be a line over $Z$ whose restriction to some twistor sphere has Chern number $1$ and let $\nabla^H$ be the connection on 
the direct image $H$ by $\p$ of $\psi^*L$\,, determined by $L$ and a connection on $Y$ compatible with $\psi$\, 
(recall Remark \ref{rem:compatible_ro-connection}\,).\\ 
\indent 
Denote by $\nabla^F$ the connection on $F$ given by the fundamental monopole with respect to $L$ and the decomposition $\widetilde{F}=H^*\oplus F$ 
corresponding to $\nabla^H$, where recall that $F$ is the vector bundle over $M$ such that $TM=H\otimes F$. 

\begin{thm}[compare \cite{Sal-dg_qm}\,,\,\cite{AleMarPon-99}\,] \label{thm:q_torsion_flat} 
The connection $\nabla^H\otimes\nabla^F$ is torsion free. Conversely, any torsion free connection on $M$, compatible with the underlying 
almost quaternionic structure of $M$, is obtained this way. 
\end{thm} 
\begin{proof} 
Let $P$ be the frame bundle of $H$ and let $\p_P:P\to M$ be its projection. 
Then $P$ is a hypercomplex manifold and its Obata connection $\nabla^P$ is the tensor product of the trivial connection 
on $P\times\C^{\!2}$ and the pull back by $\p_P$ of the fundamental monopole of $M$ with respect to $L$\,.\\ 
\indent 
After passing to ${\rm GL}(2)$-quotients, the inclusion into $TP$ of the `horizontal' distribution on $P$, corresponding to $\nabla^H$, 
gives the inclusion $TM=H\otimes F\subseteq H\otimes\widetilde{F}=\widetilde{T}M$ (see Remark \ref{rem:about_P}(2)\,).\\ 
\indent 
From the fact that $\p_P$ is twistorial and from Remark \ref{rem:ro-hyper_from_ro-quatern}\,, we deduce that the torsion of 
$\nabla^H\otimes\nabla^F$ is given by the horizontal part of the torsion of $\nabla^P$ evaluated on horizontal vector fields. 
But $\nabla^P$ is torsion free and, hence, also, $\nabla^H\otimes\nabla^F$ is torsion free.\\ 
\indent 
Now, the converse statement is an immediate consequence of the fact that both the compatible connections on $Y$ and 
the torsion free connections on $M$, compatible with the underlying almost quaternionic structure of $M$, 
are affine spaces over the space of $1$-forms on $M$.  
\end{proof} 

\subsection{Hypercomplex and hyper-K\"ahler manifolds} In Theorem \ref{thm:quatern_flat} and with the same notations as in \eqref{e:for_tilde}\,, 
if $M$ is a hypercomplex manifold, assertion (i) is equivalent 
to the fact that the connection on $F$, given by the Ward transformation applied to $\chi^*\bigl(\ol(-1)\bigr)\otimes({\rm ker}\dif\!\chi)$\,, is flat. 
Indeed, from the exact sequence appearing in the proof of Corollary \ref{cor:BirGro}\,, it follows that \eqref{e:for_tilde} splits, the fundamental 
monopole of $M$ (with respect to $\chi^*\bigl(\ol(1)\bigr)$\,) induces a connection on $F$, and $H^*=\widetilde{F}/F$ is endowed with a flat connection. 
(Consequently, \emph{the connection on $H$ induced by the fundamental monopole of $M$ and the projection $\widetilde{F}=H^*\oplus F\to H^*$ is flat}.)
Therefore the connection of $\widetilde{F}$ is flat if and only if the connection of $F$ is flat; equivalently, 
the Obata connection of $M$ is flat. The flat model is $\C^{\!2n}$ identified with the space of projective lines in $\C\!P^{n+1}$ which are 
disjoint from a fixed projective subspace of codimension $2$\,.\\ 
\indent 
A manifold is hyper-K\"ahler if and only if it is hypercomplex and endowed with a Riemannian metric preserved by the Obata connection. 
If $M$ is hypercomplex then any hyper-K\"ahler metric on it corresponds, under the Ward transformation, with a nondegenerate section 
of $\Lambda^2\bigl(\chi^*\bigl(\ol(-1)\bigr)\otimes({\rm ker}\dif\!\chi)\bigr)^*$; in particular, the dimension of $M$ is divisible by $4$\,. 
From the fact that the Obata connection is torsion free, we deduce that any such section restricts to give symplectic structures on the fibres of $\chi$\,, 
a well known fact \cite{HiKaLiRo}\,. 

\subsection{Quaternionic-K\"ahler manifolds} \label{subsection:qK} 
With notations as in Theorem \ref{thm:quatern_flat} (but without any flatness assumption) 
a contact structure on $Z$ can be defined as a nonzero section $\theta$ of $\mathscr{T}^*Z\otimes\mathcal{L}$ such that $\theta(\1)=0$ 
and $\dif^{\nabla}\!\theta$ is nondegenerate, where $\nabla$ is the canonical $\r_Z$-connection on $\mathcal{L}$\,, with $\r_Z:\mathscr{T}Z\to TZ$ 
the projection. (This, obviously, does not require $Z$ to be a twistor space.)\\ 
\indent 
We shall assume, by passing to an open subset of $M$, if necessary, that $\theta$ restricted to any twistor sphere $t\subseteq Z$ 
induces an isomorphism between $Tt$ and $\mathcal{L}|_t$\,. Equivalently, we assume the twistor spheres 
transversal to the contact distribution $\H=\r_Z({\rm ker}\,\theta)$\,.\\ 
\indent  
Similarly as before, by passing to an open neighbourhood of each point of $M$ we may find a line bundle $L$ on $Z$ such that $L^2=\mathcal{L}$\,; 
in particular, $L$ restricted to any twistor sphere has Chern number $1$\,. 
Then $\dif^{\nabla}\!\theta$ defines a linear symplectic structure on $L^*\otimes\mathscr{T}Z$, and 
therefore the fundamental monopole preserves a linear symplectic structure on $\widetilde{F}$. 
Moreover, as $\theta$ induces a contact structure on each twistor sphere 
(the corresponding contact distributions are just the trivial zero distributions), also, 
$H^*$ is endowed with a linear symplectic structure and the embedding $H^*\subseteq\widetilde{F}$ preserves the linear symplectic structures. 
Therefore \eqref{e:for_tilde} splits, as we may identify $F$ with the symplectic orthogonal complement of $H^*$ in $\widetilde{F}$.\\ 
\indent 
Now, just recall that $TM=H\otimes F$ and, similarly to the hyper-K\"ahler case, we obtain the Riemannian metric $g_{\theta}$ on $M$. 
To describe the Levi-Civita connection of $g_{\theta}$ we, again, proceed similarly to the hyper-K\"ahler case:  
endow $F$ with the connection $\nabla^F$ induced from $L^*\otimes\H$ by the Ward transformation, and recall that the splitting of \eqref{e:for_tilde} 
corresponds to a connection $\nabla^H$ on $H$ which, obviously, preserves the linear symplectic structure of $H^*$. 

\begin{thm}[compare \cite{War-graviton}\,,\,\cite{LeB-qK_twist}\,] \label{thm:contact_qK} 
The connection $\nabla^H\otimes\nabla^F$ is the Levi-Civita connection of $g_{\theta}$\,, and, thus, $(M,g_{\theta})$ is quaternionic-K\"ahler. 
\end{thm} 
\begin{proof} 
Note that, we may also describe the linear symplectic 
structure of $F$ as follows. Firstly, at each point, $\r_Z^{-1}(\H)$ is equal to the symplectic orthogonal complement of $\1$ in $\mathscr{T}Z$. 
Hence, the restriction of $\dif^{\nabla}\!\theta$ to $\r_Z^{-1}(\H)$ descends to a linear symplectic structure on $L^*\otimes\H$ which quite clearly 
corresponds to the linear symplectic structure of $F$. In particular, $\nabla^F$ preserves the linear symplectic structure of $F$.\\ 
\indent 
The second description of the linear symplectic structure of $F$ also shows that the connection induced by $\nabla^H$ on $Y\bigl(=P(H^*)\bigr)$ 
is (given by) $(\dif\!\psi)^{-1}(\H)$\,, where recall that $\psi:Y\to Z$ is the submersion giving the twistor space $Z$. 
Furthermore, on denoting by $\check{\theta}$ 
the $L^2$-valued $1$-form on $Z$ such that $\theta=\check{\theta}\circ\r_Z$\,, we have that $\check{\theta}$ lifts 
to $Y$ as the composition of the projection $TY\to{\rm ker}\dif\!\p$\,, with kernel $(\dif\!\psi)^{-1}(\H)$\,, followed by an isomorphism 
from ${\rm ker}\dif\!\p$ onto $\psi^*\bigl(L^2\bigr)$\,. It follows quickly that $\nabla^H$ is the connection determined by $L$ and the connection it induces 
on $Y$.\\   
\indent 
Thus, $\nabla^H\otimes\nabla^F$ preserves $g_{\theta}$ and, together with Theorem \ref{thm:q_torsion_flat}\,, this completes the proof. 
\end{proof}    
 
\indent 
The infinitesimal automorphisms of a contact structure on $Z$, given by $\theta:\mathscr{T}Z\to\mathcal{L}$\,, 
can be characterised as sections $u$ of $\mathscr{T}Z$ such that $\Lie^{\nabla}_{\,u}\theta=0$\,, where $\nabla$ is the canonical $\r_Z$-connection 
of $\mathcal{L}$\,; equivalently (as it is more familiar), the local flow of $\r_Z(u)$ preserves 
the contact distribution $\H\bigl(=\r({\rm ker}\,\theta)\bigr)$\,. 
Note that, if $u$ is an infinitesimal automorphism of $\theta$ then 
\begin{equation} \label{e:kernel_of_dif} 
(\dif^{\nabla}\!\theta)(u,v)=-\nabla_v\bigl(\theta(u)\bigr)\;, 
\end{equation} 
for any $v\in\mathscr{T}Z$.\\  
\indent   
Let $V$ be a (finite dimensional) vector space of infinitesimal automorphisms of $\theta$. Let  
$s_{\theta,V}$ be the section of $\mathcal{L}\otimes V^*$ given by $s_{\theta,V}(z,u)=\theta(u_z)$\,, for any $z\in Z$ and $u\in V$. 
Then the zero set of $s_{\theta,V}$ is \emph{the base locus of $\mathcal{L}$\,, with respect to $\theta(V)$}\,, and in its complement 
the differential of the map $\phi_{\theta,V}:Z\to P(V^*)$ induced by $s_{\theta,V}$ is given by 
$\nabla s_{\theta,V}:\mathscr{T}Z\to\mathcal{L}\otimes V^*$ (here, $\nabla$, also, denotes the tensor product 
of the canonical $\r_Z$-connection on $\mathcal{L}$ and the trivial connection on $Z\times V^*$). 
Thus, from \eqref{e:kernel_of_dif} we deduce that, at each $z\in Z$, the kernel of $\dif\!\phi_{\theta,V}$ is the image by $\r_Z$ 
of the symplectic orthogonal complement of $\{u_z\,|\,u\in V\}$ (essentially, a well known fact).\\ 
\indent 
Suppose, now, that $Z$ is the twistor space of the quaternionic-K\"ahler manifold $(M,g_{\theta})$\,. We may suppose that each $u\in V$ 
is such that the local flow of $\r_Z(u)$ maps twistor spheres to twistor spheres and thus it corresponds to a Killing vector 
field $X^u$ on $(M,g_{\theta})$\,. 
In fact, by applying Theorem \ref{thm:global_ro_M} we may obtain $X^u$ as the image through $\widetilde{\r}$ of the direct image by $\p$ 
of $\psi^*u$\,.\\ 
\indent 
Similarly, $s_{\theta,V}$ corresponds to a section $S_{\theta,V}$ of $({\rm End}_0H)\otimes V^*$ obtained as follows, 
where ${\rm End}_0$ denotes the space of trace free endomorphisms. 
Firstly, recall that $P$ admits a reduction $P_0$ to ${\rm SL}(2)$ corresponding to the symplectic structure of $H$ induced 
by $\theta$ (and $L$). Also, we have a decomposition 
$TP_0/{\rm SL}(2)=({\rm End}_0H)\oplus TM$ corresponding to $\nabla^H$ of Theorem \ref{thm:contact_qK}\,. Further, for any $u\in V$, 
the direct image by $\p$ of $\psi^*u$ restricts to give a section $\widetilde{X}^u$ of $TP_0/{\rm SL}(2)$\,. Then $S_{\theta,V}(x,u)$ 
is the projection to ${\rm End}_0(H_x)$ of $\widetilde{X}^u_{\,x}$\,.\\ 
\indent 
Now, the \emph{base locus} of $S_{\theta,V}$ is formed of the points $x\in M$ where $S_{\theta,V}(x)$ seen as an element of ${\rm Hom}({\rm End}_0(H_x),V^*)$ 
has rank less than $3$\,. Thus, outside its base locus, $S_{\theta,V}$ determines a map $\Phi_{\theta,V}:M\to{\rm Gr}_3(V^*)$ 
associating to any $x\in M$ the image of $S_{\theta,V}(x)$ seen as linear map from ${\rm End}_0(H_x)$ to $V^*$. 
 Furthermore, similarly to the line bundles case, the differential of $\Phi_{\theta,V}$ is determined by 
 $\nabla S_{\theta,V}:TP_0/{\rm SL}(2)\to({\rm End}_0H)\otimes V^*$, 
where $\nabla$ is the $\widetilde{\r}$-connection on ${\rm End}_0H$, induced by the canonical $\widetilde{\r}$-connection of $H$,  
tensorised with the trivial connection on $M\times V^*$. (Note that, these ideas, also, provide an alternative proof for Theorem \ref{thm:quatern_flat}\,, 
without involving the twistor space of $M$.)

\begin{thm}[compare \cite{Sal-Inventiones1982}\,,\,\cite{LeB-qK_twist}\,,\,\cite{AleMar-Annali96}\,] \label{thm:degen_qK} 
Let $M$ be a quaternionic manifold with twistor space $Z$, given by $\psi:Y\to Z$. If the fibres of $\psi$ are connected 
then there exists a natural correspondence between the following:\\ 
\indent 
\quad{\rm (i)} Pairs $(g,\nabla)$ formed of a section of $\odot^2T^*M$ and a torsion free connection on $M$ 
both compatible with its quaternionic structure and such that $\nabla g=0$\,.\\ 
\indent 
\quad{\rm (ii)} Distributions $\H$ on $Z$ of corank $1$ and transversal to the twistor spheres.\\ 
\indent 
Moreover, for any such $(g,\nabla)$ and $\H$, with nontrivial $g$ (equivalently, nonintegrable $\H$), 
the kernel of $g$ gives a foliation on $M$ locally defined by 
twistorial submersions $\phi$ onto quaternionic-K\"ahler manifolds $M_{\phi}$ such that:\\ 
\indent 
\quad{\rm (1)} $\phi$ preserves, by pull back and an obvious Lie groups morphism, the metrics and the corresponding connections;\\ 
\indent 
\quad{\rm (2)} the differential of the (local) submersion between the twistor spaces, corresponding to $\phi$\,, maps $\H$ 
onto the contact distribution of the twistor space of $M_{\phi}$\,.     
\end{thm} 
\begin{proof}
A pair as in (i) gives a distribution as in (ii) after a straightforward argument by involving the first Bianchi identity.\\ 
\indent 
How to pass from (ii) to (i)\,, and the second statement follows quickly from the proof of Theorem \ref{thm:contact_qK}\,. 
\end{proof} 
 
\indent 
In Theorem \ref{thm:degen_qK}\,, the connectedness assumption on the fibres of $\psi$ is necessary only when passing from (i) to (ii)\,. 
Another such sufficient assumption is $M$ be the complexification of a `real' quaternionic manifold.\\ 
\indent 
Examples of twistorial submersions as in Theorem \ref{thm:degen_qK} can be found in \cite{IMOP}\,. Also, in \cite[\S5]{fq_2} a construction 
of distributions as in Theorem \ref{thm:degen_qK} can be found.

\section{A construction of quaternionic manifolds} \label{section:4} 

\indent 
In this section we consider $\r$-quaternionic manifolds (of constant type and) with equivariant \cite{Pan-hmhKPW} normal sequences 
(of the twistor spheres). This means that the exact sequence 
\begin{equation} \label{e:consistent_normal_sequence} 
0\longrightarrow{\rm ker}\dif\!\p\longrightarrow\psi^*(TZ)\longrightarrow\mathcal{T}M\longrightarrow0 
\end{equation}  
(obtained from $0\longrightarrow{\rm ker}\dif\!\p\longrightarrow TY\longrightarrow\p^*(TM)\longrightarrow0$ 
by taking quotients over ${\rm ker}\dif\!\psi$\,) is invariant under the automorphism bundle ${\rm Aut}\,Y$\,; in particular, ${\rm Aut}\,Y$ 
is embedded into the extended automorphism bundle $\mathcal{P}$ of \eqref{e:consistent_normal_sequence} as a section of the 
obvious bundle morphism from $\mathcal{P}$ onto ${\rm Aut}\,Y$.\\ 
\indent 
This implies that a distinguished $\r$-connection exists, on the frame bundle of $Y$, giving $\psi$ (rather the 
foliation it determines). Moreover, this $\r$-connection is ${\rm CO}(3)$-invariant 
and, accordingly, the `typical vector bundle' (giving the linear $\r$-quaternionic structure at each point of $M$\,) is ${\rm CO}(3)$-invariant.\\   
\indent 
That is, we shall be led to consider $\r$-quaternionic manifolds $M$ of constant type and structural group $G$ which is a Lie subgroup 
of the automorphism group of the typical vector bundle.  
Also, the $\r$-connection on $Y$ giving $\psi$ is $G$-invariant. This means that the restriction 
to each fibre of $\p$ of the diagram  
\begin{equation} \label{e:for_ro_connection}   
\begin{gathered} 
\xymatrix{
       &                         \hspace{1mm}                     \mathcal{E}     \ar[dl] \ar[d]                       \\
                      \psi^*(TZ)   \ar[r]                            &     \mathcal{T}M                                                         
   } 
\end{gathered} 
\end{equation} 
is $G$-invariant. Note that, for (classical) connections this is automatically satisfied (for the $\r$-quaternionic manifolds 
of constant type, with structural group $G$) as then $\psi^*(TZ)=({\rm ker}\dif\!\psi)\oplus\mathcal{T}M$.\\ 
\indent 
To describe the relevant `typical diagram' 
we start with the Riemann sphere and choose a square root of its tangent bundle which we denote by $\ol(1)$\,, 
and let $U_n$ be the space of sections of $\ol(n)$\,, $n\in\mathbb{N}\setminus\{0\}$\,. 
Thus, we may identify the Riemann sphere with $P(U_1^*)$ and $\ol(-1)$ with the tautological line bundle over it. 
Furthermore, we have $\bigl(T(\ol(2)\setminus0)\bigr)/(\C\!\setminus\{0\})=\ol(1)\otimes U_1^*$, and,  
consequently, we have identified $\mathfrak{sl}(U_1)\,(=U_2)$ and the Lie algebra of vector fields on the Riemann sphere.  
Moreover, as $\bigl(T(\ol(n)\setminus0)\bigr)/(\C\!\setminus\{0\})=\bigl(T(\ol(2)\setminus0)\bigr)/(\C\!\setminus\{0\})$\,, 
we, also, retrieve the fact that $U_n$ is the irreducible representation space of dimension $n+1$ of $\mathfrak{sl}(U_1)$\,, 
$n\in\mathbb{N}\setminus\{0\}$\,. Obviously, this is the simplest case of the Borel--Weil theorem which we recalled, also, to emphasize 
the following exact sequence of $\mathfrak{sl}(U_1)$-invariant vector bundles 
$$0\longrightarrow\ol(-1)\otimes U_{n-1}\to\ol\otimes U_n\to\ol(n)\longrightarrow0\;,$$   
for any $n\in\mathbb{N}\setminus\{0\}$ (and where $U_0=\C$). This implies that the restriction of $TPU_n$ to the Veronese curve 
$PU_1\,(=P(U_1^*))\subseteq PU_n$ given by $\ol(n)$ is $U_{n-1}\otimes\ol(n+1)$\,. It follows that the (equivariant) normal exact sequence of 
$PU_1\subseteq PU_n$ is \cite{Pan-inf_auto_princ}\,: 
\begin{equation} \label{e:Veronese_normal_sequence} 
0\longrightarrow\ol(2)\longrightarrow U_{n-1}\otimes\ol(n+1) \longrightarrow U_{n-2}\otimes\ol(n+2)\longrightarrow0\;. 
\end{equation} 

\begin{defn} \label{defn:Veronese_space}  
A \emph{Veronese space} is a $\r$-quaternionic manifold with equivariant normal sequences given by \eqref{e:Veronese_normal_sequence}\,. 
\end{defn} 
 
\indent 
We are, thus, interested in the twistor theory associated to the $\mathfrak{sl}(U_1)$-invariant vector bundle $U_{n-2}\otimes\ol(n+2)$\,, 
where $n\geq2$\,. Pointwisely, this is given by the representation space $U_{n-2}\otimes U_{n+2}$ and the 
linear $\r$-quaternionic structure given by the $\mathfrak{sl}(U_1)$-invariant projection 
$\r_n:U_1\otimes U_{n-2}\otimes U_{n+1}\to U_{n-2}\otimes U_{n+2}$\,,  
corresponding to the $\mathfrak{sl}(U_1)$-invariant morphism of vector bundles from $U_{n-2}\otimes U_{n+1}\otimes\ol(1)$ 
onto $U_{n-2}\otimes\ol(n+2)$\,. The automorphism group of this linear twistorial structure is ${\rm CO}(3)=(\C\setminus\{0\})\times{\rm PGL}(U_1)$\,.\\ 
\indent 
Now, we can give the following typical diagram: 
\begin{equation} \label{e:typical_diagram_for_Veronese}   
\begin{gathered} 
\xymatrix{
       &                         \hspace{1mm}                    U_{n-2}\otimes U_{n+1}\otimes\ol(1)    \ar[dl] \ar[d]                       \\
                      U_{n-1}\otimes\ol(n+1)   \ar[r]                            &    U_{n-2}\otimes\ol(n+2)                                                       
   } 
\end{gathered} 
\end{equation} 
where the down-left arrow is given by $U_{n-1}\otimes U_n\otimes\ol(1)\to U_{n-1}\otimes\ol(n+1)$ 
(induced by a particlar case of \eqref{e:E}\,) and the isomorphism $U_{n-1}\otimes U_n=(U_{n-2}\otimes U_{n+1})\oplus U_1$\,. 

\begin{thm} \label{thm:Veronese_space} 
The projective frame bundle of the twistor space of a Veronese space is the twistor space of a quaternionic manifold. 
\end{thm} 
\begin{proof} 
The projective frame bundle of the twistor space of a Veronese space $M$ corresponds, under the Ward transformation, to the principal bundle 
$Q$ over $M$ associated to the principal bundle of $Y$ through the Lie groups morphism ${\rm PGL}(U_1)\to{\rm PGL}(U_{n-1})$\,.\\ 
\indent  
As an $\mathfrak{sl}(U_1)$-representation space, the Lie algebra of ${\rm PGL}(U_{n-1})$ is $U_{n-2}\otimes U_n$\,.\\ 
\indent  
Together with \eqref{e:typical_diagram_for_Veronese}\,, this implies that the underlying almost $\r$-quaternionic structure of $Q$ 
is induced by an isomorphism between $TQ/{\rm PGL}(U_{n-1})$ and the quaternionic vector bundle giving the almost $\r$-quaternionic 
structure of $M$.   
\end{proof} 

\indent 
For $n=2$\,, Theorem \ref{thm:Veronese_space} can be strengthened as follows. 

\begin{thm} \label{thm:Veronese_space_n=2} 
Let $Z$ be a surface endowed with an embedded Riemann sphere $t\subseteq Z$ with nontrivial normal sequence and 
whose normal bundle has Chern number $4$\,.\\ 
\indent 
Then an open neighbourhood of $t$ in $Z$ is the twistor space of a Veronese space for which $t$ is a twistor sphere. 
Consequently, the projective frame bundle of that neighbourhood is the twistor space of a quaternionic manifold of dimension $8$\,.  
\end{thm}  
\begin{proof}
By passing to an open neighbourhood of $t$ in $Z$ we may suppose the following:\\ 
a) $Z$ is the twistor space of a $\r$-quaternionic manifold for which $t$ is a twistor sphere;\\ 
b) the normal sequence of each twistor sphere in $Z$ is nontrivial;\\ 
c) there exists a line bundle $L$ over $Z$ such that $(\psi^{*\!}L)^2$ and $(\psi^{*\!}L)^4$ are isomorphic to 
${\rm ker}\dif\!\p$ and $\mathcal{T}M$, respectively.\\ 
\indent 
Let $H$ be the direct image by $\p$ of $L$\,. By tensorising \eqref{e:consistent_normal_sequence} 
with the dual of $(\psi^{*\!}L)^3$ we deduce that $\psi^*(TZ)=H\otimes(\psi^{*\!}L)^3$ and the proof is complete. 
\end{proof}

\begin{rem} 
The quaternionic manifold of Theorem \ref{thm:Veronese_space_n=2} is the principal bundle of $Y$. Thus, this, also, 
admits a (distinct) $\r$-hypercomplex structure (see Remark \ref{rem:ro-hypercomplex_first}\,).  
\end{rem} 

\indent 
The space of Veronese curves (of degree $n$) is a Veronese space. This is the flat model for the tame Veronese spaces 
with only one fundamental monopole; moreover, similarly to Theorem \ref{thm:quatern_flat}\,, we have a result 
for such tame $\r$-quaternionic manifolds.\\ 
\indent 
Then Theorem \ref{thm:Veronese_space} endow $({\rm PGL}(U_n)\times{\rm PGL}(U_{n-1}))/{\rm PGL}(U_1)$\,, $(n\geq2)$\,,  
with quaternionic structures (\,\cite{Pan-qgfs}\,; see \cite{Pan-hmhKPW}\,). 

\begin{exm} 
There are two conjugations on the Riemann sphere induced by the antipodal map and the symmetry in a plane through the origin 
(the latter, obviously, unique up to conjugations), respectively. These lead, for example, to real quaternionic and 
to paraquaternionic manifolds, respectively.\\ 
\indent 
Consequently, we obtain the following, where $n\geq1$ for (1)\,, and $n\geq2$ for (2) and (3)\,:\\ 
\indent 
\quad(1) real quaternionic structures on $({\rm SL}(2n+1,\R)\times({\rm GL}(n,\Hq)\cap\,{\rm SL}(2n,\C)))/{\rm SU}(2)$\,;\\ 
\indent 
\quad(2) real quaternionic structures on $(({\rm GL}(n,\Hq)\cap\,{\rm SL}(2n,\C))\times{\rm SL}(2n-1,\R))/{\rm SU}(2)$\,;\\ 
\indent 
\quad(3) paraquaternionic structures on $({\rm SL}(n+1,\R)\times{\rm SL}(n,\R))/{\rm SL}(2,\R)$\,.\\    
\indent 
On the other hand, for any $n\geq2$\,, the conjugation on $({\rm PGL}(U_n)\times{\rm PGL}(U_{n-1}))/{\rm PGL}(U_1)$ giving, up to covering spaces, 
$({\rm SU}(n+1)\times{\rm SU}(n))/{\rm SU}(2)$ is not twistorial, and therefore the latter does not inherit, this way, 
a real (or para)quaternionic structure. This corrects a statement of \cite{Pan-qgfs}\,. 
\end{exm}

\section{On twistorial harmonic morphisms with one-dimensional fibres} \label{section:5}  

\indent 
In this section, although some of the results hold in more generality, for simplicity, $M$ will denote a real(-analytic) quaternionic-K\"ahler or hyper-K\"ahler manifold. 
To unify the notations of this and the previous sections, one just have to replace in the latter $M$ with $M^{\C}$. Further, as, now, $\psi$ 
restricted to $\p^{-1}(M)$ is a diffeomorphism, the diagram giving the twistor space simplifies, as it is well know, to a fibration whose total space 
and projection we will denote by $Y$ and $\p$, respectively (instead of $\p^{-1}(M)$ and $\p|_{\p^{-1}(M)}$\,, respectively). 
Note that, $Y$ embeds into the complex Grassmannian of $TM$ such that any point of $Y$ is an isotropic space of dimension $2k$ tangent to $M$, 
where $\dim M=4k$, with $k\in\mathbb{N}\setminus\{0\}$\,.\\ 
\indent 
The main general reference for harmonic morphisms is \cite{BaiWoo2}\,; see, also, \cite{Pan-BookR} and \cite{BookW} for more recent results. 

\begin{prop} \label{prop:twist_wp} 
Let $\phi:M\to N$ be a harmonic morphisms of warped-product type, with one-dimensional fibres, and, locally, let $N\subseteq M$ be a horizontal section. 
Denote by $\r$ the restriction to $TM|_N$ of $\dif\!\phi$\,.\\ 
\indent 
Then $(N,TM|_N,\r)$ is a $\r$-quaternionic manifold, where $TM|_N$ is endowed with the restriction of the Levi-Civita connection of $M$.  
\end{prop} 
\begin{proof} 
For any $y\in Y|_N$ we have that $y\cap T^{\C\!}N$ is an isotropic space of dimension $2k-1$\,.   
Consequently, $y^{\perp}\subseteq T^{\C\!}N$ is coisotropic. Therefore $(TM|_N,\r)$ is an almost $\r$-quaternionic structure on $N$,  
and, from the integrability result of \cite{Pan-q_integrab} and by applying \cite[Lemma 5.1]{Pan-1d} (see \cite[Appendix A.2]{Pan-BookR} and \cite[Chapter 11]{BaiWoo2}\,), 
we obtain that this is integrable. 
\end{proof} 

\indent 
Presumably, the following result is not new. We omit the proof. 

\begin{prop} \label{prop:Killing_qK_quatern}
Any Killing vector field on a quaternionic-K\"ahler manifold $M$ is quaternionic; consequently, it corresponds to a holomorphic vector field on the twistor space of $M$ 
preserving the contact distribution.  
\end{prop} 

\indent 
We end with the following result. 

\begin{thm} \label{thm:1d_qK} 
Let $\phi:M\to N$ be a harmonic morphism with one-di\-men\-sio\-nal fibres which is twistorial; that is, $\phi$ corresponds to a one-dimensional holomorphic foliation $\F$ on the 
twistor space $Z$ of $M$. Let $\H$ be the holomorphic distribution on $Z$ given by the metric of $M$.\\ 
\indent 
Then either $\F\subseteq\H$ or $\F$ is generated by a holomorphic vector field preserving $\H$. 
\end{thm} 
\begin{proof} 
This follows from Proposition \ref{prop:twist_wp} and \cite{PanWoo-d} (see \cite[Chapter 3]{Pan-BookR}\,). 
\end{proof}

\end{document}